\documentclass{amsart}
\usepackage{tikz,hyperref,stmaryrd}

\theoremstyle{plain}
\newtheorem{theorem}{Theorem}
\newtheorem{lemma}{Lemma}
\newtheorem{proposition}{Proposition}

\title[Thue--Morse--Sturmian words and critical bases]{Thue--Morse--Sturmian words and critical bases for ternary alphabets}

\author{Wolfgang Steiner}
\address{IRIF, CNRS UMR 8243, Universit\'e Paris Diderot -- Paris 7, Case 7014, 75205 Paris Cedex 13, FRANCE}
\email{steiner@irif.fr}
\thanks{This work was supported by the Agence Nationale de la Recherche through the project \emph{CODYS} (ANR-18-CE40-0007).}

\begin{document}
\begin{abstract}
The set of unique $\beta$-expansions over the alphabet $\{0,1\}$ is trivial for $\beta$ below the golden ratio and uncountable above the Komornik--Loreti constant. 
Generalisations of these thresholds for three-letter alphabets were studied by Komornik, Lai and Pedicini (2011, 2017).
We use $S$-adic words including the Thue--Morse word (which defines the Komornik--Loreti constant) and Sturmian words (which characterise generalised golden ratios) to determine the value of a certain generalisation of the Komornik--Loreti constant to three-letter alphabets.
\end{abstract}
\maketitle

\section{Introduction and main results}
For a base $\beta > 1$ and a sequence of digits $u_1 u_2 \cdots \in A^\infty$, with $A \subset \mathbb{R}$, let
\[
\pi_\beta(u_1 u_2\cdots) = \sum_{k=1}^\infty \frac{u_k}{\beta^k}. 
\]
We say that $u_1 u_2 \cdots$ is a $\beta$-expansion of this number.
This paper deals with \emph{unique $\beta$-expansions over~$A$}, that is with
\[
U_\beta(A) = \{\mathbf{u} \in A^\infty:\, \pi_\beta(\mathbf{u}) \ne \pi_\beta(\mathbf{v})\ \mbox{for all}\ \mathbf{v} \in A^\infty \setminus \{\mathbf{u}\}\}.
\]
We know from \cite{DaroczyKatai93} that $U_\beta(\{0,1\})$ is trivial if and only if $\beta \le \frac{1+\sqrt{5}}{2}$, where trivial means that $U_\beta(\{0,1\}) = \{\overline{0}, \overline{1}\}$, $\overline{a}$ being the infinite repetition of~$a$.
Therefore, 
\[
\mathcal{G}(A) = \inf\{\beta > 1:\, |U_\beta(A)| > 2\}
\]
is called \emph{generalised golden ratio} of~$A$.
By \cite{GlendinningSidorov01}, the set $U_\beta(\{0,1\})$ is uncountable if and only if $\beta$ is larger than the Komornik--Loreti constant $\beta_{\mathrm{KL}} \approx 1.787$; we call
\[
\mathcal{K}(A) = \inf\{\beta > 1:\, U_\beta(A)\ \mbox{is uncountable}\}
\]
\emph{generalised Komornik--Loreti constant} of~$A$.
(We can replace \emph{uncountable} throughout the paper by \emph{has the cardinality of the continuum}.)
The precise structure of $U_\beta(\{0,1\})$ was described in \cite{KomornikKongLi17}. 
For integers $M \ge 2$, $\mathcal{G}(\{0,1,\ldots,M\})$ was determined by \cite{Baker14}, and $U_\beta(\{0,1,\ldots,M\})$ was described in \cite{KongLiLudeVries17,AlcarazBarreraBakerKong19}.

For $x, y \in \mathbb{R}$, $x \ne 0$, we have $(xu_1+y_1)(xu_2+y_2)\cdots \in U_\beta(xA+y)$ if and only if $u_1u_2\cdots \in U_\beta(A)$, thus $\mathcal{G}(xA+y) = \mathcal{G}(A)$ and $\mathcal{K}(xA+y) = \mathcal{K}(A)$.
Hence, the only two-letter alphabet to consider is $\{0,1\}$, and we can restrict to $\{0,1,m\}$, $m \in (1,2]$, for three-letter alphabets; another possibility is $m \ge 2$ as in \cite{KomornikLaiPedicini11}.
We write
\[
U_\beta(m) = U_\beta(\{0,1,m\}), \quad \mathcal{G}(m) = \mathcal{G}(\{0,1,m\}), \quad \mathcal{K}(m) = \mathcal{K}(\{0,1,m\}).
\]
It was established in \cite{KomornikLaiPedicini11,Lai11,BakerSteiner17} that the generalised golden ratio $\mathcal{G}(m)$ is given by mechanical words, i.e., Sturmian words and their periodic counterparts; in particular, we can restrict to sequences $\mathbf{u} \in \{0,1\}^\infty$. 
Calculating $\mathcal{K}(m)$ seems to be much harder since this restriction is not possible.
Therefore, we study
\[
\mathcal{L}(m) = \inf\{\beta > 1:\, U_\beta(m) \cap \{0,1\}^\infty\ \mbox{is uncountable}\},
\]
following~\cite{KomornikPedicini17}, where this quantity was determined for certain intervals.
We give a complete characterisation in Theorem~\ref{t:main} below.

To this end, we use the substitutions (or morphisms)
\begin{align*}
L:\ & 0 \mapsto 0, & M:\ & 0 \mapsto 01, & R:\ & 0 \mapsto 01, \\
& 1 \mapsto 01, & & 1 \mapsto 10, & & 1 \mapsto 1,
\end{align*} 
which act on finite and infinite words by $\sigma(u_1u_2\cdots) = \sigma(u_1)\sigma(u_2)\cdots$. 
The monoid generated by a set of substitutions~$S$ (with the usual product of substitutions) is denoted by~$S^*$. 
An infinite word $\mathbf{u}$ is a \emph{limit word} of a sequence of substitutions $(\sigma_n)_{n\ge1}$ (or an \emph{$S$-adic word} if $\sigma_n \in S$ for all $n \ge 1$) if there is a sequence of words $(\mathbf{u}^{(n)})_{n\ge1}$ with $\mathbf{u}^{(1)} = \mathbf{u}$, $\mathbf{u}^{(n)} = \sigma_n(\mathbf{u}^{(n+1)})$ for all $n \ge 1$.
The sequence $(\sigma_n)_{n\ge1}$ is \emph{primitive} if for each $k \ge 1$ there is an $n \ge k$ such that both words $\sigma_k\sigma_{k+1}\cdots\sigma_n(0)$ and $\sigma_k\sigma_{k+1}\cdots\sigma_n(1)$ contain both letters $0$ and~$1$.
For $S = \{L,M,R\}$, this means that there is no $k \ge 1$ such that $\sigma_n = L$ for all $n \ge k$ or $\sigma_n = R$ for all $n \ge k$.
Let $\mathcal{S}_S$ be the set of limit words of primitive sequences of substitutions in~$S^\infty$.
Then $\mathcal{S}_{\{L,R\}}$ consists of \emph{Sturmian words}, and $\mathcal{S}_{\{M\}}$ consists of the \emph{Thue-Morse word} $0\mathbf{u} = 0110100110010110\cdots$, which defines the Komornik--Loreti constant by $\pi_{\beta_{\mathrm{KL}}}(\mathbf{u}) = 1$, and its reflection by $0\leftrightarrow1$. 
We call the elements of~$\mathcal{S}_{\{L,M,R\}}$, which to our knowledge have not been studied yet, \emph{Thue--Morse--Sturmian words}. 
For details on $S$-adic and other words, we refer to \cite{Lothaire02,BertheDelecroix14}.

For $\mathbf{u} \in \{0,1\}^\infty$ and $m \in (1,2]$, define $f_\mathbf{u}(m)$ (if $\mathbf{u}$ contains at least two ones) and $g_\mathbf{u}(m)$ as the unique positive solutions of 
\[
f_\mathbf{u}(m)\, \pi_{f_\mathbf{u}(m)}(\sup O(\mathbf{u})) = m
\quad \mbox{and} \quad (g_\mathbf{u}(m)-1)(1+\pi_{g_\mathbf{u}(m)}(\inf O(\mathbf{u}))) = m
\]
respectively, where $O(u_1 u_2 \cdots) = \{u_k u_{k+1} \cdots:\, k \ge 1\}$ denotes the shift orbit and infinite words are ordered by the \emph{lexicographic order}.
For the existence and monotonicity properties of $f_\mathbf{u}(m)$ and $g_\mathbf{u}(m)$, see Lemma~\ref{l:fg} below.
We define $\mu_\mathbf{u}$ by 
\[
f_\mathbf{u}(\mu_\mathbf{u}) = g_\mathbf{u}(\mu_\mathbf{u}),
\]
i.e., $f_\mathbf{u}(\mu_\mathbf{u}) = g_\mathbf{u}(\mu_\mathbf{u}) = \beta$ with $\beta\, \pi_\beta(\sup O(\mathbf{u})) = (\beta-1)(1+\pi_\beta(\inf O(\mathbf{u})))$.

The main result of \cite{KomornikLaiPedicini11} can be written as 
\[
\mathcal{G}(m) = \begin{cases}f_{\sigma(\overline{0})}(m) & \mbox{if}\ m \in [\mu_{\sigma(1\overline{0})}, \mu_{\sigma(\overline{0})}],\, \sigma \in \{L,R\}^* M,\\[.5ex]
g_{\sigma(\overline{0})}(m) & \mbox{if}\ m \in [\mu_{\sigma(\overline{0})}, \mu_{\sigma(0\overline{1})}],\,  \sigma \in \{L,R\}^* M, \\[.5ex]
f_{\overline{1}}(m) & \mbox{if}\ m \in [\mu_{0\overline{1}}, 2], \\[.5ex]
1+\sqrt{m} & \mbox{if}\ m = \mu_{\mathbf{u}},\, \mathbf{u} \in \mathcal{S}_{\{L,R\}}; \end{cases}
\]
cf.\ \cite[Proposition~3.18]{BakerSteiner17}, where substitutions $\tau_h = L^h R$ are used and $f,g,\mu,\mathcal{S}$ are defined slightly differently.
Our main theorem looks similar, but we need $\{L,M,R\}$ instead of $\{L,R\}$, and the roles of $f$ and $g$ are exchanged. 

\begin{theorem} \label{t:main}
The function $\mathcal{L}(m)$, $1 < m \le 2$, is given by
\[
\mathcal{L}(m) = \begin{cases}
g_{\sigma(1\overline{0})}(m) & \mbox{if}\ m \in [\mu_{\sigma(1\overline{0})}, \mu_{\sigma(01\overline{0})}],\, \sigma \in \{L,M,R\}^* M, \\[.5ex]
f_{\sigma(0\overline{1})}(m) & \mbox{if}\ m \in [\mu_{\sigma(10\overline{1})}, \mu_{\sigma(0\overline{1})}],\, \sigma \in \{L,M,R\}^* M, \\[.5ex]
g_{0\overline{1}}(m) & \mbox{if}\ m \in [\mu_{0\overline{1}},2], \\[.5ex]
f_\mathbf{u}(m) & \mbox{if}\ m = \mu_\mathbf{u},\, \mathbf{u} \in \mathcal{S}_{\{L,M,R\}}.
\end{cases}
\]
The Hausdorff dimension of $\pi_\beta(U_\beta(m))$ is positive for all $\beta > \mathcal{L}(m)$. 
\end{theorem}

The graphs of $\mathcal{G}(m)$ and $\mathcal{L}(m)$ are drawn in Figure~\ref{f:GL}.
For example, $\sigma = M$ gives
\[
\mathcal{L}(m) = \left\{\begin{array}{ll}g_{0\overline{01}}(m) & \mbox{if}\ m \in [\mu_{0\overline{01}}, \mu_{110\overline{01}}] \approx [1.281972,1.46811], \\
f_{1\overline{10}}(m) & \mbox{if}\ m \in [\mu_{001\overline{10}}, \mu_{1\overline{10}}] \approx [1.516574,1.55496].\end{array}\right.
\]
Taking $\sigma = M^2$, we have $\sigma(0) = 0110$, $\sigma(1) = 1001$, and
\[
\mathcal{L}(m) = \left\{\begin{array}{ll}g_{001\overline{0110}}(m) & \mbox{if}\ m \in [\mu_{001\overline{0110}}, \mu_{1101001\overline{0110}}] \approx [1.47571,1.503114], \\
f_{110\overline{1001}}(m) & \mbox{if}\ m \in [\mu_{0010110\overline{1001}},  \mu_{110\overline{1001}}] \approx [1.504152,1.509304].\end{array}\right.
\] 
Subintervals of the first three intervals were also given by \cite{KomornikPedicini17}. 

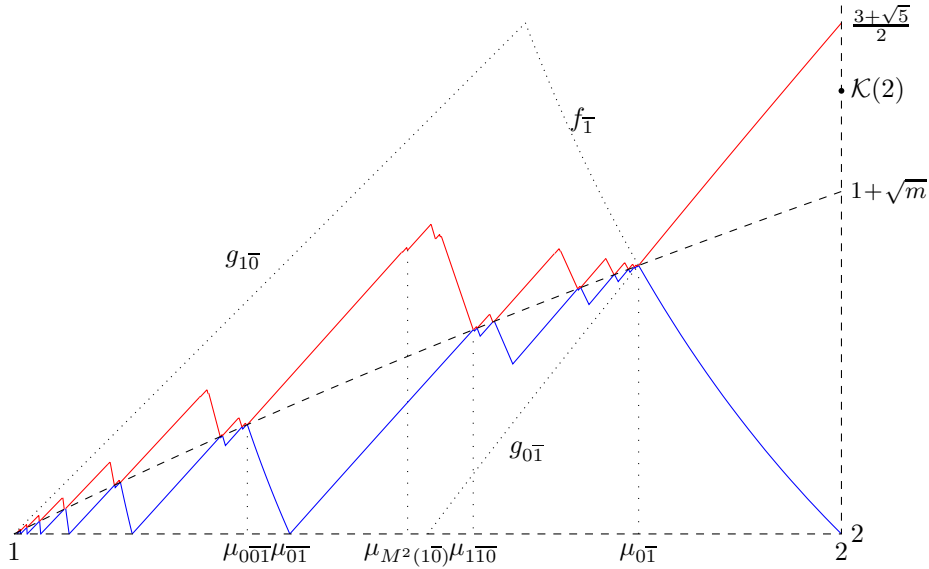
\begin{figure}[ht]
\begin{tikzpicture}[scale=11]
\begin{scope}[dotted]
\draw[domain=1.5:1.7549] plot (\x, {(\x+1+sqrt(\x^2+2*\x-3))/2});  
\node at (1.62,2.1){$g_{0\overline{1}}$};
\draw(1,2)--node[above left]{$g_{1\overline{0}}$}(1.618,2.618);
\draw[domain=1.618:1.7549] plot (\x, {\x/(\x-1)}); 
\node[right] at (1.66,2.5){$f_{\overline{1}}$};
\end{scope}

\begin{scope}[loosely dotted]
\draw(1.282,2.1322)--(1.282,2)node[below]{$\mu_{0\overline{01}}$}; 
\node[below] at (1.3333,2){$\mu_{\overline{01}}$}; 
\draw(1.4757,2.3432)--(1.4757,2)node[below]{$\mu_{M^2(1\overline{0})}$}; 
\draw(1.555,2.247)--(1.555,2)node[below]{$\mu_{1\overline{10}}$}; 
\draw(1.7549,2.3247)--(1.7549,2)node[below]{$\mu_{0\overline{1}}$}; 
\end{scope}

\begin{scope}[blue]
\draw[domain=1.00389:1.00392] plot (\x, {(\x/(\x-1))^(1/8)}); %LLLLLL
\draw(1.00392,2)--(1.0077,2.0039);

\draw[domain=1.00777:1.00787] plot (\x, {(\x/(\x-1))^(1/7)}); %LLLLL
\draw(1.0079,2)--(1.0154,2.0077);

\draw[domain=1.0155:1.0159] plot (\x, {(\x/(\x-1))^(1/6)}); %LLLL
\draw(1.0159,2)--(1.0305,2.0152);

\draw[domain=1.031:1.0323] plot (\x, {(\x/(\x-1))^(1/5)}); %LLL
\draw(1.0323,2)--(1.0607,2.0299);

\draw[domain=1.06247:1.06667] plot (\x, {(\x/(\x-1))^(1/4)}); %LL
\draw(1.0667,2)--(1.1213,2.0589);

\draw(1.1217,2.0591)--(1.1223,2.056)--(1.1279,2.062); %LLR

\draw[domain=1.1287:1.14286] plot (\x, {(\x/(\x-1))^(1/3)}); %L
\draw(1.14286,2)--(1.2492,2.1177);

\draw(1.2495,2.1178)--(1.2499,2.1166)--(1.252,2.1189); %LRL

\draw(1.252,2.1189)--(1.2555,2.1069)--(1.2755,2.1294); %LR

\draw(1.2756,2.1294)--(1.2764,2.1269)--(1.2806,2.1316); %LRR

\draw[domain=1.3333:1.555] plot (\x, {\x/2+sqrt(\x*\x/4+\x)}); %
\draw[domain=1.282:1.3333] plot (\x, {sqrt(\x/(\x-1))});

\draw(1.5559,2.24735)--(1.5566,2.2455)--(1.5596,2.24885); %RLL

\draw(1.5597,2.2489)--(1.5634,2.2392)--(1.5787,2.2565); %RL

\draw(1.5805,2.2572)--(1.6028,2.20557)--(1.68125,2.2966); %R

\draw(1.6815,2.2967)--(1.682,2.295)--(1.6848,2.298); %RRL

\draw(1.6849,2.298)--(1.6949,2.2775)--(1.7252,2.3135); %RR

\draw(1.7258,2.3137)--(1.7302,2.3052)--(1.7425,2.32); %RRR

\draw(1.7426,2.3201)--(1.7445,2.3165)--(1.7496,2.3227); %RRRR

\draw(1.7496,2.3227)--(1.7505,2.3212)--(1.7526,2.3239); %RRRRR

\draw(1.7526,2.3239)--(1.753,2.3232)--(1.7539,2.3244); %RRRRRR

\draw[domain=1.7549:2] plot (\x, {\x/(\x-1)}); 
\end{scope}

\begin{scope}[red]
\draw(1.0039, 2.0019)--(1.0077, 2.0057) (1.0077, 2.0058)--(1.0077,2.0039); %LLLLLL

\draw(1.0078, 2.0039)--(1.0152, 2.0114) (1.0152, 2.0114)--(1.0154,2.0077); %LLLLL

\draw(1.0155, 2.0077)--(1.0298, 2.0223) (1.0300, 2.0224)--(1.0305,2.0152); %LLLL

\draw(1.0310, 2.0154)--(1.0582, 2.0434) (1.0590, 2.0438)--(1.0607,2.0299); %LLL

\draw(1.0625, 2.0308)--(1.1134, 2.0843) (1.1164, 2.0857)--(1.1213,2.0589); %LL
\draw(1.1136, 2.0844)--(1.1160, 2.0869) (1.1160, 2.0869)--(1.1162, 2.0856); %LLM

\draw(1.1217, 2.0591)--(1.1272, 2.0650) (1.1272, 2.0650)--(1.1279, 2.0620); %LLR

\draw(1.1287, 2.0624)--(1.2241, 2.1653) (1.2354, 2.1702)--(1.2492,2.1177); %L
\draw(1.2242, 2.1653)--(1.2250, 2.1662) (1.2250, 2.1662)--(1.2251, 2.1657); %LML
\draw(1.2251, 2.1657)--(1.2331, 2.1743) (1.2332, 2.1744)--(1.2344, 2.1697); %LM

\draw(1.2495, 2.1178)--(1.2516, 2.1201) (1.2516, 2.1201)--(1.2520, 2.1189); %LRL

\draw(1.2520, 2.1189)--(1.2714, 2.1403) (1.2719, 2.1406)--(1.2755, 2.1294); %LR

\draw(1.2756, 2.1294)--(1.2797, 2.1341) (1.2798, 2.1341)--(1.2806, 2.1316); %LRR

\draw(1.2820, 2.1322)--(1.4681, 2.3402) (1.5166, 2.3596)--(1.5550, 2.2470); %
\draw(1.4684, 2.3403)--(1.4693, 2.3413) (1.4693, 2.3413)--(1.4695, 2.3407); %MLL
\draw(1.4695, 2.3407)--(1.4746, 2.3464) (1.4746, 2.3464)--(1.4755, 2.3432); %ML
\draw(1.4757, 2.3432)--(1.5031, 2.3736) (1.5042, 2.3741)--(1.5093, 2.3567); %M
\draw(1.5031, 2.3737)--(1.5040, 2.3746) (1.5040, 2.3746)--(1.5041, 2.3740); %MM
\draw(1.5095, 2.3568)--(1.5143, 2.3622) (1.5143, 2.3622)--(1.5153, 2.3591); %MR
\draw(1.5153, 2.3591)--(1.5162, 2.3601) (1.5162, 2.3601)--(1.5163, 2.3595); %MRR

\draw(1.5559, 2.2474)--(1.5589, 2.2507) (1.5589, 2.2507)--(1.5597, 2.2488); %RLL

\draw(1.5597, 2.2489)--(1.5746, 2.2655) (1.5750, 2.2656)--(1.5790, 2.2556); %RL

\draw(1.5805, 2.2572)--(1.6531, 2.3391) (1.6607, 2.3420)--(1.6812, 2.2966); %R
\draw(1.6537, 2.3393)--(1.6587, 2.3450) (1.6588, 2.3450)--(1.6602, 2.3418); %RM

\draw(1.6815, 2.2967)--(1.6840, 2.2996) (1.6840, 2.2996)--(1.6848, 2.2979); %RRL

\draw(1.6849, 2.2980)--(1.7142, 2.3320) (1.7155, 2.3325)--(1.7252, 2.3135); %RR
\draw(1.7142, 2.3320)--(1.7152, 2.3331) (1.7152, 2.3331)--(1.7155, 2.3325); %RRM

\draw(1.7258, 2.3137)--(1.7379, 2.3280) (1.7381, 2.3281)--(1.7425, 2.3200); %RRR

\draw(1.7426, 2.3201)--(1.7476, 2.3262) (1.7477, 2.3262)--(1.7496, 2.3227); %RRRR

\draw(1.7496, 2.3227)--(1.7518, 2.3254) (1.7518, 2.3254)--(1.7526, 2.3239); %RRRRR

\draw(1.7526, 2.3239)--(1.7535, 2.3250) (1.7535, 2.3250)--(1.7539, 2.3244); %RRRRRR

\draw[domain=1.7549:2] plot (\x, {(\x+1+sqrt(\x^2+2*\x-3))/2}); 
\end{scope}

\draw[dashed](1,2)node[below]{$1$}--(2,2)node[right]{$2$}--(2,2.65);
\draw[dashed,domain=1:2] plot (\x, {1+sqrt(\x)}) node[right]{\small$1\!+\!\sqrt{m}$};
\node[below] at (2,2){$2$}; 
\node[right] at (2,2.618){$\frac{3+\sqrt{5}}{2}$};

\fill(2,2.53595) circle(.1pt);
\node[right] at (2,2.53595){$\mathcal{K}(2)$};
\end{tikzpicture}
\caption{The critical bases $\mathcal{G}(m)$ (below $1+\sqrt{m}$, blue) and $\mathcal{L}(m)$ (above $1+\sqrt{m}$, red).} \label{f:GL}
\end{figure}

By \cite{KomornikLaiPedicini11,KomornikPedicini17}, we have, for all $m \in (1,2]$, 
\[
2 \le \mathcal{G}(m) \le 1+\sqrt{m} \le \mathcal{K}(m) \le \mathcal{L}(m) \le g_{1\overline{0}}(m) = 1+m,
\]
with $\mathcal{G}(m) = \mathcal{L}(m)$ if and only if $m \in \{\mu_{\sigma(1\overline{0})}, \mu_{\sigma(0\overline{1})}\}$, $\sigma \in \{L,R\}^* M$, or $m = \mu_\mathbf{u}$, $\mathbf{u} \in \mathcal{S}_{\{L,R\}}$.
Besides those~$m$, the value of $\mathcal{K}(m)$ is known only for $m = 2$ from \cite{deVriesKomornik09,AlloucheFrougny09,KomornikPedicini17}, with $\mathcal{K}(2) \approx 2.536 < \frac{3+\sqrt{5}}{2} = \mathcal{L}(2)$. 
The functions $\mathcal{G}(m)$, $\mathcal{K}(m)$ and $\mathcal{L}(m)$ are continuous for $m > 1$ by~\cite{KomornikLaiPedicini11,KomornikPedicini17}; at least for the generalised golden ratio, this also holds for larger alphabets by \cite{BakerSteiner17}. 

\section{Proof of the main theorem}\label{sec:proof-main-theorem}
We first establish relations between $f_\mathbf{u}(m)$, $g_\mathbf{u}(m)$ and $\mathbf{u} \in U_\beta(m)$.
For convenience, we write $\inf(\mathbf{u})$ for $\inf O(\mathbf{u})$ and $\sup(\mathbf{u})$ for $\sup O(\mathbf{u})$ in the following. 

\begin{lemma} \label{l:U}
Let $m \in (1,2]$, $\beta \in (1,1+m]$.
For $\mathbf{u} \in \{0,1\}^\infty$, we have $\mathbf{u} \in U_\beta(m)$ if and only if  $0\mathbf{u} \in U_\beta(m)$.
For $\mathbf{u} \in 1\{0,1\}^\infty \setminus \{1\overline{0}\}$, $\mathbf{u} \in U_\beta(m)$ implies that $\beta \ge \max(f_\mathbf{u}(m), g_\mathbf{u}(m))$, and $\beta > \max(f_\mathbf{u}(m), g_\mathbf{u}(m))$ implies that $\mathbf{u} \in U_\beta(m)$.
\end{lemma}

\begin{proof}
For $\beta \in (1,1+m]$, $\mathbf{u} = u_1 u_2\cdots \in \{0,1,m\}^\infty$, $x \in [0,\frac{m}{\beta-1}]$, we have $\pi_\beta(\mathbf{u}) = x$ if and only if $u_k = d(T^{k-1}(x))$ for all $k \ge 1$, with the branching $\beta$-transformation
\[
T:\, [0,\tfrac{m}{\beta-1}] \to [0,\tfrac{m}{\beta-1}], \ x \mapsto \beta x - d(x), \ d(x) = \left\{\hspace{-.5em}\begin{array}{cl}0 & \hspace{-.5em} \mbox{if}\ x < \frac{1}{\beta}, \\ 0\ \mbox{or}\ 1 & \hspace{-.5em} \mbox{if}\ \frac{1}{\beta} \le x \le \frac{m}{\beta(\beta-1)}, \\ 1 & \hspace{-.5em} \mbox{if}\ \frac{m}{\beta(\beta-1)} < x < \frac{m}{\beta}, \\ 1\ \mbox{or}\ m & \hspace{-.5em} \mbox{if}\ \frac{m}{\beta} \le x \le \frac{1}{\beta} + \frac{m}{\beta(\beta-1)}, \\ m & \hspace{-.5em} \mbox{if}\ x > \frac{1}{\beta} + \frac{m}{\beta(\beta-1)},\end{array}\right.
\]
see Figure~\ref{f:T}.
We have thus 
\[
\mathbf{u} \in U_\beta(m) \ \Longleftrightarrow\ \pi_\beta(u_k u_{k+1}\cdots) \notin [\tfrac{1}{\beta},\tfrac{m}{\beta(\beta-1}]  \cup [\tfrac{m}{\beta}, \tfrac{1}{\beta}+\tfrac{m}{\beta(\beta-1)}]\ \mbox{for all}\ k \ge 1.
\]
For $\mathbf{u} \in \{0,1\}^\infty \setminus \{\overline{0}\}$, this means that $\beta > 2$ and 
\[
\pi_\beta(u_k u_{k+1}\cdots) < \tfrac{m}{\beta},\ \pi_\beta(u_{k+1} u_{k+2} \cdots) > \tfrac{m}{\beta-1} - 1\ \mbox{for all $k \ge 1$ such that $u_k = 1$}, 
\]
see \cite[Lemma~3.9]{BakerSteiner17}, i.e., 
\begin{gather*}
\beta\, \pi_\beta(\sup(\mathbf{u})) \le m \le (\beta-1)(1+\pi_\beta(\inf\nolimits_1(\mathbf{u}))), \\
\inf\nolimits_1(u_1u_2\cdots) = \inf\{u_{k+1}u_{k+2}\cdots:\, k\ge 1, u_k=1\},
\end{gather*}
with strict equalities if the supremum and infimum are attained.
In particular, we have $\mathbf{u} \in U_\beta(m)$ if and only if $0\mathbf{u} \in U_\beta(m)$.
If $\mathbf{u}$ starts with~$1$, then $\inf_1(\mathbf{u}) = \inf(\mathbf{u})$, and the first lines of Lemma~\ref{l:fg} below conclude the proof of the lemma.
\end{proof}

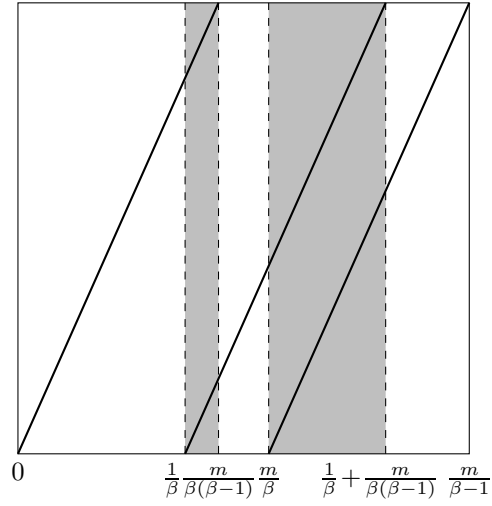
\begin{figure} \label{f:T}
\begin{tikzpicture}[scale=5]
\fill[black!25](.4444,0)--(.5333,0)--(.5333,1.2)--(.4444,1.2) (.6666,0)--(.9777,0)--(.9777,1.2)--(.6666,1.2);
\draw[dashed](.4444,0)--(.4444,1.2) (.5333,0)--(.5333,1.2) (.6666,0)--(.6666,1.2) (.9777,0)--(.9777,1.2);
\draw(0,0)node[below]{$0$}--(1.2,0)node[below]{$\frac{m}{\beta-1}\vphantom{\frac{1}{\beta}}$}--(1.2,1.2)--(0,1.2)--cycle;
\draw[thick](0,0)--(.5333,1.2) (.4444,0)--(.9777,1.2) (.6666,0)--(1.2,1.2);
\node[below] at (.41,0){$\frac{1}{\beta}$};
\node[below] at (.5333,0){$\frac{m}{\beta(\beta-1)}\vphantom{\frac{1}{\beta}}$};
\node[below] at (.6666,0){$\frac{m}{\beta}\vphantom{\frac{1}{\beta}}$};
\node[below] at (.96,0){$\frac{1}{\beta}\!+\!\frac{m}{\beta(\beta-1)}$};
\end{tikzpicture}
\caption{The branching $\beta$-transformation $T$ for $\beta = 9/4$, $m=3/2$.}
\end{figure}

\begin{lemma} \label{l:fg}
Let $m \in (1,2]$, $\mathbf{u}, \mathbf{u}' \in \{0,1\}^\infty$.
Then $g_\mathbf{u}(m)$ is well defined.
If $\mathbf{u}$ contains at least two ones, then $f_\mathbf{u}(m)$ and $\mu_\mathbf{u}$ are well defined, and we have
\begin{gather*}
\max(f_\mathbf{u}(m), g_\mathbf{u}(m)) \ge 2, \\
\beta\, \pi_\beta(\sup(\mathbf{u})) < m \quad \mbox{for all}\ \beta > f_\mathbf{u}(m), \\
(\beta-1)(1+\pi_\beta(\inf(\mathbf{u})) > m \quad \mbox{for all}\ \beta > g_\mathbf{u}(m), \\
f_\mathbf{u}(m) > f_\mathbf{u}(m') \quad \mbox{and} \quad g_\mathbf{u}(m) < g_\mathbf{u}(m') \quad \mbox{if}\ m < m', \\
f_\mathbf{u}(m) < f_{\mathbf{u}'}(m) \quad \mbox{if}\ \sup(\mathbf{u}) < \sup(\mathbf{u}')\ \mbox{and}\ f_\mathbf{u}(m) \ge 2, \\
g_\mathbf{u}(m) > g_{\mathbf{u}'}(m) \quad \mbox{if}\ \inf(\mathbf{u}) < \inf(\mathbf{u}')\ \mbox{and}\ g_{\mathbf{u}'}(m) \ge 2.
\end{gather*}
\end{lemma}

\begin{proof}
Let $\sup(\mathbf{u}) = \mathbf{v}$ and set $h_\mathbf{v}(x,m) = x \pi_x(\mathbf{v}) -m$.
Then $h_\mathbf{v}(x,m)$ is strictly decreasing in $x$ and~$m$ (for $x > 1$).
If $\mathbf{u}$ contains at least two ones, then this also holds for~$\mathbf{v}$, thus $\lim_{x\to1}h_\mathbf{v}(x,m) \ge 2-m$ and $\lim_{x\to\infty} h_\mathbf{v}(x,m) = 1-m$.
Therefore, there is, for each $m \in (1,2]$, a unique $x_{m,\mathbf{v}} \ge 1$ such that $h_\mathbf{v}(x_{m,\mathbf{v}},m) = 0$, i.e., $f_\mathbf{u}(m) = x_{m,\mathbf{v}}$.
If $m < m'$, then we have $x_{m,\mathbf{v}} > x_{m',\mathbf{v}}$. 
If $\mathbf{v} < \mathbf{v}'$ and $x \ge 2$, then we have $h_\mathbf{v}(x,m) < h_{\mathbf{v}'}(x,m)$, thus $x_{m,\mathbf{v}} < x_{m,\mathbf{v'}}$ if $x_{m,\mathbf{v}} \ge 2$.

Let now $\inf(\mathbf{u}) = \mathbf{v}$ and set $h_\mathbf{v}(x,m) = \frac{m}{x-1} - \pi_x(\mathbf{v}) - 1$.
Since $\frac{m}{x-1} = \pi_x(\overline{m})$, $h_\mathbf{v}(x,m)$ is strictly decreasing in $x$ (for $x > 1$) and strictly increasing in~$m$.
Again, there is, for each $m \in (1,2]$, a unique $x_{m,\mathbf{v}} > 1$ such that $h_\mathbf{v}(x_{m,\mathbf{v}},m) = 0$, i.e., $g_\mathbf{u}(m) = x_{m,\mathbf{v}}$.
We have $h_\mathbf{v}(x,m) < 0$ if $x > x_{m,\mathbf{v}}$, $x_{m,\mathbf{v}} < x_{m',\mathbf{v}}$ if $m < m'$, and $h_\mathbf{v}(x,m) > h_{\mathbf{v}'}(x,m)$ if $\mathbf{v} < \mathbf{v}'$, $x \ge 2$, thus $x_{m,\mathbf{v}} > x_{m,\mathbf{v}'}$ if $x_{m,\mathbf{v}'} \ge 2$.

Since $f_\mathbf{u}(m)$ is strictly decreasing, $g_\mathbf{u}(m)$ is strictly increasing, $\lim_{m\to1} f_\mathbf{u}(m) = \infty$, $f_\mathbf{u}(2) \le 2$ and $g_\mathbf{u}(2) \ge 2$, we have $f_\mathbf{u}(m) = g_\mathbf{u}(m)$ for a unique $m \in (1,2]$. 

Let $\beta = f_\mathbf{u}(\mu_\mathbf{u}) = g_\mathbf{u}(\mu_\mathbf{u})$, i.e., $\beta\, \pi_\beta(\sup(\mathbf{u})) = (\beta-1)(1+\pi_\beta(\inf(\mathbf{u})))$.
We have $\sup(\mathbf{u}) \ge 1\inf(\mathbf{u})$. 
If equality holds, then $\beta = 2$.
Otherwise, $\sup(\mathbf{u})$ starts with $1v_1\cdots v_{k-1}1$ and $\inf(\mathbf{u})$ starts with $v_1\cdots v_{k-1}0$ for some $v_1\cdots v_{k-1}$, $k\ge 1$.
Then 
\[
\beta \pi_\beta(\sup(\mathbf{u})) \ge 1 + \sum_{i=1}^{k-1} \frac{v_i}{\beta^i} + \frac{1}{\beta^k},\ (\beta-1)(1+\pi_\beta(\inf(\mathbf{u}))) \le (\beta-1) \bigg(1 + \sum_{i=1}^{k-1} \frac{v_i}{\beta^i}\bigg) + \frac{1}{\beta^k},
\]
thus $\beta \ge 2$. 
By the monotonicity properties that are proved above, this implies that $\max(f_\mathbf{u}(m), g_\mathbf{u}(m)) \ge 2$ for all $m \in (1,2]$.
\end{proof}

Therefore, it is crucial to determine $\inf(\mathbf{u})$ and $\sup(\mathbf{u})$. 
We set
\[
\sup\nolimits_0(u_1u_2\cdots) = \sup\{u_{k+1}u_{k+2}\cdots:\, k\ge 1, u_k=0\},
\]
similarly to $\inf_1(u_1u_2\cdots) = \inf\{u_{k+1}u_{k+2}\cdots:\, k\ge 1, u_k=1\}$. 

\begin{lemma} \label{l:infsup}
For all $\mathbf{u} \in \{0,1\}^\infty$, we have
\[
\inf(L(\mathbf{u})) = L(\inf(\mathbf{u})), \quad \inf(R(\mathbf{u})) = R(\inf(\mathbf{u})), \quad
0\sup(L(\mathbf{u})) = L(\sup(\mathbf{u})).
\]
If $\inf(\mathbf{u}) = \inf_1(\mathbf{u})$, then $\inf(M(\mathbf{u})) = 0 M(\inf(\mathbf{u}))$.
If $\sup(\mathbf{u}) = \sup_0(\mathbf{u})$, then 
\[
\sup(R(\mathbf{u})) = 1 R(\sup(\mathbf{u})), \quad \sup(M(\mathbf{u})) = 1 M(\sup(\mathbf{u})).
\]

For each $\sigma \in \{L,M,R\}^*$, there is a suffix $w$ of $\sigma(1)$ such that $\inf_1(\sigma(\mathbf{u})) = \inf(\sigma(\mathbf{u})) = w \sigma(\inf(\mathbf{u}))$ for all $\mathbf{u} \in \{0,1\}^\infty$ with $\inf(\mathbf{u}) = \inf_1(\mathbf{u})$.

For each $\sigma \in \{L,M,R\}^*M \cup \{L,M,R\}^*R$, there is a suffix $w$ of $\sigma(0)$ such that $\sup_0(\sigma(\mathbf{u})) =\sup(\sigma(\mathbf{u})) = w \sigma(\sup(\mathbf{u}))$ for all $\mathbf{u} \in \{0,1\}^\infty$ with $\sup(\mathbf{u}) = \sup_0(\mathbf{u})$. 

For each $\sigma \in \{L,M,R\}^*L$, there is a prefix $w$ of $\sigma(\overline{0})$ such that $w \sup_0(\sigma(\mathbf{u})) = w \sup(\sigma(\mathbf{u})) = \sigma(\sup(\mathbf{u}))$ for all $\mathbf{u} \in \{0,1\}^\infty$ with $\sup(\mathbf{u}) = \sup_0(\mathbf{u})$. 
\end{lemma}

\begin{proof}
The first statements follow from the facts that $L,M,R$ are order-preserving on infinite words and that $\inf(\mathbf{u}) = \inf_1(\mathbf{u})$, $\sup(\mathbf{u}) = \sup_0(\mathbf{u})$ mean that $1 \inf(\mathbf{u})$, $0 \sup(\mathbf{u})$ are in the closure of~$O(\mathbf{u})$.

We claim that, for each $\sigma \in \{L,M,R\}^*$, there is a suffix $1w$ of $\sigma(1)$ such that $\inf_1(\sigma(\mathbf{u})) = \inf(\sigma(\mathbf{u})) = w \sigma(\inf(\mathbf{u}))$ for all $\mathbf{u} \in \{0,1\}^\infty$ with $\inf(\mathbf{u}) = \inf_1(\mathbf{u})$.
If $1w$ is a suffix of $\sigma(1)$, then $1L(w)$, $10M(w)$ and $1R(w)$ are suffixes of $L\sigma(1)$, $M\sigma(1)$ and $R\sigma(1)$ respectively.
Therefore, this claim holds for $L\sigma$, $M\sigma$ and $R\sigma$ when it holds for $\sigma$. 
Since it holds for $\sigma = \mathrm{id}$, it holds for all $\sigma \in \{L,M,R\}^*$.

Next we claim that, for each $\sigma \in \{L,M,R\}^* \{M,R\}$, there is a suffix $01w$ of $\sigma(0)$ such that $\sup_0(\sigma(\mathbf{u})) = \sup(\sigma(\mathbf{u})) = 1w \sigma(\sup(\mathbf{u}))$ for all $\mathbf{u} \in \{0,1\}^\infty$ with $\sup(\mathbf{u}) = \sup_0(\mathbf{u})$.
This holds for $\sigma \in \{M,R\}$.
If $01w$ is a suffix of $\sigma(0)$, then $01L(w)$, $01M(1w)$ and $01R(1w)$ are suffixes of $L\sigma(0)$, $M\sigma(0)$ and $R\sigma(0)$ respectively.
Therefore, this claim holds for all $\sigma \in \{L,M,R\}^* \{M,R\}$.
 
Finally we claim that, for each $\sigma \in \{L,M,R\}^* L$, there is a prefix $w0$ of $\sigma(\overline{0})$ such that $w0 \sup_0(\sigma(\mathbf{u})) = w0 \sup(\sigma(\mathbf{u})) = \sigma(\sup(\mathbf{u}))$ for all $\mathbf{u} \in \{0,1\}^\infty$ with $\sup(\mathbf{u}) = \sup_0(\mathbf{u})$.
This holds for $\sigma = L$ .
If $w0$ is a prefix of $\sigma(\overline{0})$, then $L(w0)0$, $M(w)0$ and $R(w)0$ are prefixes of $L\sigma(\overline{0})$, $M\sigma(\overline{0})$ and $R\sigma(\overline{0})$ respectively.
Therefore, this claim holds for all $\sigma \in \{L,M,R\}^*L$.
\end{proof}

Now we can prove that Theorem~\ref{t:main} gives an upper bound for~$\mathcal{L}(m)$, cf.\ Figure~\ref{f:musigma0}.

\begin{proposition} \label{p:upper}
Let $m \in (1,2]$. 
We have 
\[
\mathcal{L}(m) \le \begin{cases}
g_{\sigma(1\overline{0})}(m) & \mbox{if}\ m \ge \mu_{\sigma(1\overline{0})},\, \sigma \in \{L,M,R\}^* M, \\[.5ex]
f_{\sigma(0\overline{1})}(m) & \mbox{if}\ m \le \mu_{\sigma(0\overline{1})},\, \sigma \in \{L,M,R\}^* M, \\[.5ex]
g_{0\overline{1}}(m) & \mbox{if}\ m \ge \mu_{0\overline{1}}, \\[.5ex]
g_\mathbf{u}(m) & \mbox{if}\ m \ge \mu_\mathbf{u},\, \mathbf{u} \in \mathcal{S}_{\{L,M,R\}}, \\
f_\mathbf{u}(m) & \mbox{if}\ m \le \mu_\mathbf{u},\, \mathbf{u} \in \mathcal{S}_{\{L,M,R\}}.
\end{cases}
\]
If $\beta$ is above this bound, then the Hausdorff dimension of $\pi_\beta(U_\beta(m))$ is positive.
\end{proposition}

\begin{figure}[ht] 
\begin{tikzpicture}[scale=2.25]
\fill[color=black!10](-2.25,1.25)--(0,-1)--(2.25,1.25);
\fill[color=black!20](-2.25,1.25)--(-1,0)--(.25,1.25);
\fill[color=black!20](-.25,1.25)--(1,0)--(2.25,1.25);
\fill[color=black!30](-.25,1.25)--(0,1)--(.25,1.25);
\draw[dashed](-1.5,0)node[left]{$1+\sqrt{m}$}--(2,0);
\node[below] at (0,0){$\sigma(\overline{0})\in U_\beta(m)$};
\node at (-1,1){$\sigma(1\overline{0})\in U_\beta(m)$};
\node at (1,1){$\sigma(0\overline{1})\in U_\beta(m)$};
\draw(-2.25,1.25)--(.125,-1.125);
\node[below=1ex] at (-2,1){$f_{\sigma(\overline{0})}$};
\draw(2.25,1.25)--(-.125,-1.125);
\node[below=2ex] at (2,1){$g_{\sigma(\overline{1})}$};
\draw(-.25,1.25)--(1.25,-.25)node[below right=-1ex]{$f_{\sigma(0\overline{1})}$};
\draw(.25,1.25)--(-1.25,-.25)node[below left=-1ex]{$g_{\sigma(1\overline{0})}$};
\draw(-.375,.875)--(1.125,-.625)node[below right=-1ex]{$f_{\sigma(01\overline{0})}$};
\draw(.375,.875)--(-1.125,-.625)node[below left=-1ex]{$g_{\sigma(10\overline{1})}$};
\draw[dotted](0,-1)--(0,-1.25)node[below=-1ex]{$\mu_{\sigma(\overline{0})}$};
\draw[dotted](-1,0)--(-1,-1.25)node[below=-1ex]{$\mu_{\sigma(1\overline{0})}$};
\draw[dotted](1,0)--(1,-1.25)node[below=-1ex]{$\mu_{\sigma(0\overline{1})}$};
\draw[dotted](.25,.75)--(.25,-1.25)node[below right=-1ex]{$\mu_{\sigma(10\overline{1})}$};
\draw[dotted](-.25,.75)--(-.25,-1.25)node[below left=-1ex]{$\mu_{\sigma(01\overline{0})}$};
\draw[thick,blue](-1,0)--node[below left=-1ex]{$\mathcal{G}(m)$}(0,-1)--node[below right=-1ex]{$\mathcal{G}(m)$}(1,0);
\draw[thick,red](-1,0)--node[above left=-1ex]{$\mathcal{L}(m)$}(-.25,.75) (1,0)--node[above right=-1ex]{$\mathcal{L}(m)$}(.25,.75);
\end{tikzpicture}
\caption{A~schematic picture for $\sigma \in \{L,R\}^* M$. For $\sigma \in \{L,M,R\}^* M$, the situation is similar, except for $\mathcal{G}(m)$ and $1+\sqrt{m}$.} \label{f:musigma0}
\end{figure}
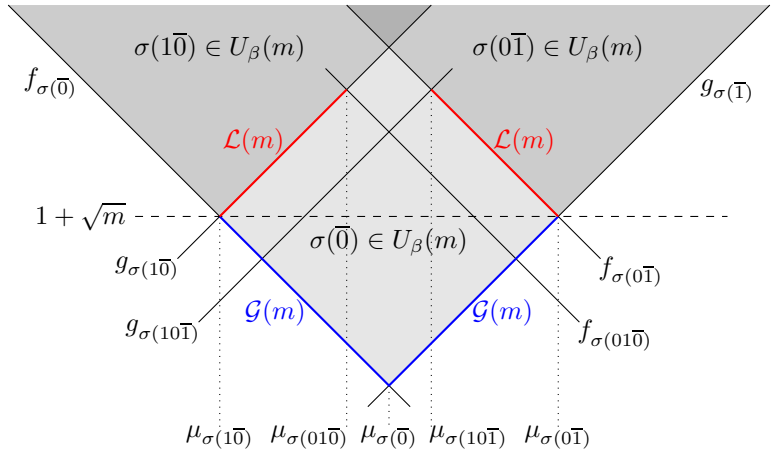

\begin{proof}
Let $\sigma \in \{L,M,R\}^*$.
For all $h \ge 1$, $\mathbf{v} \in 1\{0(01)^h, 0(01)^{h+1}\}^\infty$, we have
\[
\inf(\sigma(\mathbf{v})) \ge \inf(\sigma(\overline{10(01)^{h-1}0})) \quad \mbox{and} \quad \sup(\sigma(\mathbf{v})) \le \sup(\sigma(\overline{(01)^{h+1}0}))
\]
by Lemma~\ref{l:infsup}, with 
\[
\inf(\sigma(\overline{10(01)^{h-1}0})) \to \inf(\sigma M(1\overline{0})),\ \sup(\sigma(\overline{(01)^{h+1}0})) \to \sup(\sigma M(\overline{0})) \quad (h\to\infty).
\]
Therefore, we have for each $\beta > \max(f_{\sigma M(\overline{0})}(m), g_{\sigma M(1\overline{0})}(m))$ some $h \ge1$ such that $\sigma(\{0(01)^h, 0(01)^{h+1}\}^\infty) \subseteq U_\beta(m)$.
If $m \ge \mu_{\sigma M(1\overline{0})}$, then $f_{\sigma M(\overline{0})}(m) = f_{\sigma M(1\overline{0})}(m) \le g_{\sigma M(1\overline{0})}(m)$, thus $U_\beta(m) \cap \{0,1\}^\infty$ is uncountable (and has the cardinality of the continuum) for all $\beta > g_{\sigma M(1\overline{0})}(m)$, i.e., $\mathcal{L}(m) \le g_{\sigma M(1\overline{0})}(m)$.
By symmetry, sequences in $\sigma(\{1(10)^h, 1(10)^{h+1}\}^\infty)$ give that $\mathcal{L}(m) \le f_{\sigma M(0\overline{1})}(m)$ for $m \le \mu_{\sigma M(0\overline{1})}$.
Similarly, sequences in $1\{01^h, 01^{h+1}\}^\infty$ give that $\mathcal{L}(m) \le g_{0\overline{1}}(m)$ for $m \ge \mu_{0\overline{1}}$.

Let now $\mathbf{u}$ be a limit word of a primitive sequence $(\sigma_n)_{n\ge1} \in \{L,M,R\}^\infty$, and set $\sigma'_n = \sigma_1 \sigma_2 \cdots \sigma_n$. 
Then $\inf(\sigma'_n(1\overline{0})) \le \inf(\mathbf{u}) \le \inf(\sigma'_n(10\overline{1}))$ for all $n\ge1$, thus $\inf(\sigma'_n(1\overline{0}))
\to \inf(\mathbf{u})$ and (by symmety) $\sup(\sigma'_n(0\overline{1})) \to \sup(\mathbf{u})$ as $n \to \infty$.
Therefore, for each $\beta > \max(f_\mathbf{u}(m), g_\mathbf{u}(m))$ there is $n \ge 1$ such that $\sigma'_n(\mathbf{v}) \in U_\beta(m)$ for all $\mathbf{v} \in \{0,1\}^\infty \setminus \{\overline{0},\overline{1}\}$, hence $\mathcal{L}(m) \le g_\mathbf{u}(m)$ for $m \ge \mu_\mathbf{u}$ and  $\mathcal{L}(m) \le f_\mathbf{u}(m)$ for $m \le \mu_\mathbf{u}$.

If $\{v,w\}^\infty \subseteq  U_\beta(m)$, then by \cite{Hutchinson81} we have $\dim_H(\pi_\beta(U_\beta(m))) \ge r$, with $r > 0$ such that $\beta^{-|v|r}+\beta^{-|w|r} = 1$, where $|v|$ and $|w|$ denote the lengths of $v$ and $w$.
\end{proof}

For the lower bound, we use Lemma~\ref{l:sigma} below, which tells us that, if the orbit of a sequence satisfies inequalities that hold for all non-trivial images of $\sigma \in \{L,M,R\}^*$, then it is eventually in the image of~$\sigma$.
In particular, with $\sigma = M^n$, $n \ge 0$, this yields that $U_\beta(\{0,1\})$ is countable for all $\beta$ less than the Komornik--Loreti constant; cf.\ \cite{GlendinningSidorov01}.
First we show that the conditions of Lemma~\ref{l:infsup} are satisfied for a suffix. 

\begin{lemma} \label{l:inf1sup0}
Let $\mathbf{u} \in \{0,1\}^\infty$ with $\mathbf{u} \ne 0^k \overline{1}$ and $\mathbf{u} \ne 1^k \overline{0}$ for all $k \ge 0$.
There is a suffix $\mathbf{v}$ of $\mathbf{u}$ such that $\inf(\mathbf{v}) = \inf_1(\mathbf{v}) = \inf_1(\mathbf{u})$ and $\sup(\mathbf{v}) = \sup_0(\mathbf{v}) =  \sup_0(\mathbf{u})$.
\end{lemma}

\begin{proof}
If $\inf(\mathbf{u}) = \inf_1(\mathbf{u})$ and $\sup(\mathbf{u}) = \sup_0(\mathbf{u})$, then we can take $\mathbf{v} = \mathbf{u}$. 
Otherwise, assume that $\inf(\mathbf{u}) \ne \inf_1(\mathbf{u})$, the case $\sup(\mathbf{u}) \ne \sup_0(\mathbf{u})$ being symmetric. 
Then we have $\inf(\mathbf{u}) = \mathbf{u} = 0^k 01\mathbf{u}'$ for some $k \ge 0$, $\mathbf{u}' \in \{0,1\}^\infty \setminus \{\overline{1}\}$, 
\[
\sup\nolimits_0(\mathbf{u}) = \sup\nolimits_0(01\mathbf{u}') = \sup(01\mathbf{u}'),\ \inf\nolimits_1(\mathbf{u}) = \inf\nolimits_1(01\mathbf{u}') = \inf\nolimits_1(1\mathbf{u}') = \inf(1\mathbf{u}').
\]
If $\inf_1(01\mathbf{u}') \ne \inf(01\mathbf{u}')$, then $\mathbf{u}' = 1^n01\mathbf{u}''$ with $n\ge 0$, $\mathbf{u}''>\mathbf{u}'$, which implies that $\sup_0(\mathbf{u}) = \sup_0(1\mathbf{u}') = \sup(1\mathbf{u}')$.
Hence, we can take $\mathbf{v} = 01\mathbf{u}'$ or $\mathbf{v} = 1\mathbf{u}'$.
\end{proof}

\begin{lemma} \label{l:sigma}
Let $\mathbf{u} \in \{0,1\}^\infty\!\!$, $\sigma \in \{L,M,R\}^*$, with $\inf(\mathbf{u}) \ge \inf(\sigma(1\overline{0}))$, $\sup(\mathbf{u}) \le \sup(\sigma(0\overline{1}))$.
Then $\mathbf{u}$ ends with $\sigma(\mathbf{v})$ for some $\mathbf{v} \in \{0,1\}^\infty$ or with $\sigma'(\overline{0})$, $\sigma' \in \{L,M,R\}^*M$, $\sigma \in \sigma' \{L,M,R\}^*$. 
\end{lemma}

\begin{proof}
The statement is trivially true when $\sigma$ is the identity. 
Suppose that it holds for some $\sigma \in \{L,M,R\}^* $, let $\varphi \in \{L,M,R\}$ and $\mathbf{u} \in \{0,1\}^\infty$ with $\inf(\mathbf{u}) \ge \inf(\varphi\sigma(1\overline{0}))$, $\sup(\mathbf{u}) \le \sup(\varphi\sigma(0\overline{1}))$.

If $\varphi=L$, then $\sup(\mathbf{u}) \le \overline{10}$, thus every $1$ in $\mathbf{u}$ is followed by a~$0$, hence $\mathbf{u} = L(\mathbf{v})$ or  $\mathbf{u} = 1L(\mathbf{v})$ for some $\mathbf{v} \in \{0,1\}^\infty$.
Similary, if $\varphi=R$, then $\inf(\mathbf{u}) \ge \overline{01}$, hence $\mathbf{u} = R(\mathbf{v})$ or  $\mathbf{u} = 0R(\mathbf{v})$ for some $\mathbf{v} \in \{0,1\}^\infty$. 
If $\varphi=M$, then $\inf(\mathbf{u}) \ge 0\overline{01}$ and $\sup(\mathbf{u}) \le 1\overline{10}$.
Hence, for all $k \ge 1$, $0(01)^k$ as well as $1(10)^k$ is always followed in~$\mathbf{u}$ by $01$ or~$10$.
Since $\mathbf{u}$ contains $001$ or $110$ if $\mathbf{u} \notin \{M(\overline{0}), M(\overline{1})\}$, we obtain that $\mathbf{u}$ ends with $M(\mathbf{v})$ for some $\mathbf{v} \in \{0,1\}^\infty$. 

We can assume that $\mathbf{v} \in \{\overline{0},\overline{1}\}$ or $\inf_1(\mathbf{v}) = \inf(\mathbf{v})$ and $\sup_0(\mathbf{v}) = \sup(\mathbf{v})$, by Lemma~\ref{l:inf1sup0}.
If $\mathbf{v} \ne \overline{0}$, then we cannot have $\inf(\mathbf{v}) < \inf(\sigma(1\overline{0}))$ because this would imply that $\inf(\varphi(\mathbf{v})) < \inf(\varphi\sigma(1\overline{0}))$ by Lemma~\ref{l:infsup}. 
Similarly, we obtain that $\sup(\mathbf{v}) \le \sup(\sigma(1\overline{0}))$ if $\mathbf{v} \ne \overline{1}$.
If $\mathbf{v} = \overline{0}$, $\varphi \in \{L,R\}$, then $\inf(\varphi(\overline{0})) \ge \inf(\varphi\sigma(1\overline{0}))$ implies that $\inf(\sigma(1\overline{0})) = \overline{0}$, thus $\mathbf{v} = \sigma(\overline{0})$. 
Similarly, if $\mathbf{v} = \overline{1}$ and $\varphi \in \{L,R\}$, then $\sup(\varphi(\overline{1})) \le \sup(\varphi\sigma(0\overline{1}))$ implies that $\sup(\sigma(0\overline{1})) = \overline{1}$, thus $\mathbf{v} = \sigma(\overline{1})$. 
If $\mathbf{v} \in \{\overline{0},\overline{1}\}$, $\varphi = M$, then $\mathbf{u}$ ends with $M(\overline{0})$ since $M(\overline{1}) = 1M(\overline{0})$. 
Therefore, $\mathbf{u}$ ends with $\varphi\sigma(\mathbf{v})$ or with $\sigma'(\overline{0})$, $\sigma' \in \{L,M,R\}^*M$, $\varphi\sigma \in \sigma' \{L,M,R\}^*$. 
\end{proof}

We obtain the following lower bound for~$\mathcal{L}(m)$, cf.\ Figure~\ref{f:musigma0}.

\begin{proposition} \label{p:lower}
Let $m \in (1,2]$. 
We have $\mathcal{L}(m) \ge g_{0\overline{1}}(m)$ and
\[
\mathcal{L}(m) \ge \begin{cases}
g_{\sigma(1\overline{0})}(m) & \mbox{if}\ m \le \mu_{\sigma(01\overline{0})},\, \sigma \in \{L,M,R\}^*, \\[.5ex]
f_{\sigma(0\overline{1})}(m) & \mbox{if}\ m \ge \mu_{\sigma(10\overline{1})},\, \sigma \in \{L,M,R\}^*, \\[.5ex]
g_\mathbf{u}(m) & \mbox{if}\ m \le \mu_\mathbf{u},\, \mathbf{u} \in \mathcal{S}_{\{L,M,R\}}, \\
f_\mathbf{u}(m) & \mbox{if}\ m \ge \mu_\mathbf{u},\, \mathbf{u} \in \mathcal{S}_{\{L,M,R\}}.
\end{cases}
\]
\end{proposition}

\begin{proof}
For all $\mathbf{v} \in 1\{0,1\}^\infty \setminus \{\overline{1}\}$, we have $\inf(\mathbf{v}) \le 0\overline{1}$.
Then $\mathbf{v} \in U_\beta(m)$ implies that $\beta \ge g_{0\overline{1}}(m)$ by Lemma~\ref{l:U}, hence $\mathcal{L}(m) \ge g_{0\overline{1}}(m)$.

Suppose that $U_\beta(m) \cap \{0,1\}^\infty$ is uncountable for $\beta < g_{\sigma(1\overline{0})}(m)$, $m \le \mu_{\sigma(01\overline{0})}$, $\sigma \in \{L,M,R\}^*M$, thus $\beta < g_{\sigma(01\overline{0})}(m) \le f_{\sigma(01\overline{0})}(m)$.
Then $U_\beta(m)$ contains an aperiodic sequence $\mathbf{v} \in 1\{0,1\}^\infty$, with $f_\mathbf{v}(m) < f_{\sigma(01\overline{0})}(m)$ and $g_\mathbf{v}(m) < g_{\sigma(1\overline{0})}(m)$ by Lemma~\ref{l:U}, thus $\inf(\mathbf{v}) > \inf(\sigma(1\overline{0}))$ and $\sup(\mathbf{v}) < \sup(\sigma(01\overline{0}))$ by Lemma~\ref{l:fg}.
By Lemma~\ref{l:sigma}, $\mathbf{v}$ ends with $\sigma(\mathbf{v}')$ for some (aperiodic) $\mathbf{v}' \in \{0,1\}^\infty$, contradicting that $\sup(\mathbf{v}) < \sup(\sigma(01\overline{0}))$.
Symetrically, we get that $\mathcal{L}(m) \ge f_{\sigma(0\overline{1})}(m)$ for $m \ge \mu_{\sigma(10\overline{1})}$.

If $\mathbf{u}$ is a limit word of a primitive sequence $(\sigma_n)_{n\ge1} \in \{L,M,R\}^\infty$, then we have $\mu_{\sigma'_n(01\overline{0})} \to \mu_\mathbf{u}$ for $\sigma'_n = \sigma_1 \sigma_2 \cdots \sigma_n$ as $n \to\infty$, thus $\beta < g_\mathbf{u}(m)$, $m \le \mu_\mathbf{u}$ implies that $\beta < \min(g_{\sigma'_n(01\overline{0})}(m), f_{\sigma'_n(01\overline{0})}(m))$ for some $n \ge 1$, and we obtain as in the previous paragraph that $U_\beta(m) \cap \{0,1\}^\infty$ is at most countable.  
Therefore, we have $\mathcal{L}(m) \ge g_\mathbf{u}(m)$ and, similarly, $\mathcal{L}(m) \ge f_\mathbf{u}(m)$ for $m \ge \mu_\mathbf{u}$.
\end{proof}

Propositions~\ref{p:upper} and~\ref{p:lower} prove the formula for $\mathcal{L}(m)$ in Theorem~\ref{t:main}. 
It remains to show that this covers all $m \in (1,2]$. 

For the characterisation of~$\mathcal{G}(m)$, in \cite[Proposition~3.3]{BakerSteiner17} the partition 
\[
(\overline{0}, 0\overline{1}) = \mathcal{S}_{\{L,R\}} \cup \bigcup_{\sigma\in \{L,R\}^*} [\sigma(0\overline{01}), \sigma(\overline{01})]
\]
for intervals of sequences in $\{0,1\}^\infty$ is used, which is a consequence of the partition
\[
(\overline{0}, 0\overline{1}) = L((\overline{0}, 0\overline{1})) \cup [0\overline{01}, \overline{01}] \cup R((\overline{0}, 0\overline{1})). 
\]
We have to refine these partitions. 
For $\boldsymbol{\sigma} = (\sigma_n)_{n\ge1} \in \{L,M,R\}^\infty$, set
%\begin{align*}
\[
I_{\boldsymbol{\sigma}} = \begin{cases}\{\inf(\mathbf{u}):\, \mathbf{u}\ \mbox{is a limit word of}\ \boldsymbol{\sigma}\} & \mbox{if $\boldsymbol{\sigma}$ is primitive}, \\ \{\inf(\sigma_1\sigma_2 \cdots \sigma_n(1\overline{0}))\} & \mbox{if}\ \sigma_n\sigma_{n+1}\cdots = M\overline{L},\, n\ge1, \\ [\inf(\sigma_1\sigma_2 \cdots \sigma_n(10\overline{1})), \inf(\sigma_1\sigma_2 \cdots \sigma_n(\overline{1}))] & \mbox{if}\ \sigma_n\sigma_{n+1}\cdots = M\overline{R},\, n\ge1, \\ \emptyset & \mbox{otherwise},\end{cases} 
\]
\[
J_{\boldsymbol{\sigma}} = \begin{cases}\{\sup(\mathbf{u}):\, \mathbf{u}\ \mbox{is a limit word of}\ \boldsymbol{\sigma}\} & \mbox{if $\boldsymbol{\sigma}$ is primitive}, \\ [\sup(\sigma_1\sigma_2 \cdots \sigma_n(\overline{0})), \sup(\sigma_1\sigma_2 \cdots \sigma_n(01\overline{0}))] & \mbox{if}\ \sigma_n\sigma_{n+1}\cdots = M\overline{L}, n\ge1, \\ \{\sup(\sigma_1\sigma_2 \cdots \sigma_n(0\overline{1}))\} & \mbox{if}\ \sigma_n\sigma_{n+1}\cdots = M\overline{R},\, n\ge1, \\ \emptyset & \mbox{otherwise}.\end{cases} 
\]
%\end{align*}
Note that, for a primitive sequence~$\boldsymbol{\sigma}$, $\inf(\mathbf{u})$ as well as $\sup(\mathbf{u})$ does not depend on the limit word~$\mathbf{u}$.
We order sequences in $\{L,M,R\}^\infty$ lexicographically.

\begin{lemma} \label{l:partition}
In $\{0,1\}^\infty$, we have
\[
(\overline{0}, 0\overline{1}) = \bigcup_{\boldsymbol{\sigma} \in \{L,M,R\}^\infty} I_{\boldsymbol{\sigma}} \quad \mbox{and} \quad (1\overline{0}, \overline{1}) = \bigcup_{\boldsymbol{\sigma} \in \{L,M,R\}^\infty} J_{\boldsymbol{\sigma}}.
\]
If $\boldsymbol{\sigma} < \boldsymbol{\sigma}'$, then $\mathbf{v} < \mathbf{v}'$ for all $\mathbf{v} \in I_{\boldsymbol{\sigma}}$, $\mathbf{v}' \in I_{\boldsymbol{\sigma}'}$, and for all $\mathbf{v} \in J_{\boldsymbol{\sigma}}$, $\mathbf{v}' \in J_{\boldsymbol{\sigma}'}$.
\end{lemma}

\begin{proof}
We clearly have $I_{\boldsymbol{\sigma}} \subset (\overline{0}, 0\overline{1})$ for all $\boldsymbol{\sigma} \in \{L,M,R\}^\infty$.
For all $\sigma \in \{L,M,R\}^*$, Lemma~\ref{l:infsup} gives that $\inf(\sigma(1\overline{0})) = \inf(\sigma L(1\overline{0}))$, $\inf(\sigma L(10\overline{1})) = \inf(\sigma M(1\overline{0}))$, and we have $M(\overline{1}) = R(1\overline{0})$, $R(10\overline{1}) = 10\overline{1}$, thus
\begin{multline*}
(\inf(\sigma(1\overline{0})), \inf(\sigma(10\overline{1}))) =  (\inf(\sigma L(1\overline{0})), \inf(\sigma L(10\overline{1}))) \\
\cup \{\inf(\sigma M(1\overline{0}))\} \cup  (\inf(\sigma M(1\overline{0})), \inf(\sigma M(10\overline{1}))) \\
\cup [\inf(\sigma M(10\overline{1})), \inf(\sigma M(\overline{1}))] \cup (\inf(\sigma R(1\overline{0})), \inf(\sigma R(10\overline{1})))
\end{multline*}
(in this order). 
Inductively, we obtain that the sets~$I_{\boldsymbol{\sigma}}$ are ordered by the lexicographical order on $\{L,M,R\}^\infty$. 
Moreover, the union of sets~$I_{\boldsymbol{\sigma}}$ with $\boldsymbol{\sigma}$ ending in $M\overline{L}$ or $M\overline{R}$ covers $(\inf(1\overline{0}), \inf(10\overline{1})) = (\overline{0},0\overline{1})$, except for points lying in the intersection of nested intervals $\bigcap_{n\ge1} (\inf(\sigma_1\cdots\sigma_n(1\overline{0})), \inf(\sigma_1\cdots\sigma_n(10\overline{1})))$ for some $\boldsymbol{\sigma} = (\sigma_n)_{n\ge1} \in \{L,M,R\}^\infty$.
Since $\sigma_1 \cdots \sigma_n(\overline{0})$ is close to $\sigma_1 \cdots \sigma_n(0\overline{1})$ for large~$n$, these intervals tend to some $\mathbf{v} \in \{0,1\}^\infty$.
If $\boldsymbol{\sigma}$ is primitive, then $I_{\boldsymbol{\sigma}} = \{\mathbf{v}\}$. 
If $\sigma_{n+1}\sigma_{n+2}\cdots$ is $\overline{L}$ or $\overline{R}$, then we have $\mathbf{v} = \inf(\sigma_1\cdots\sigma_n(1\overline{0}))$ or $\mathbf{v} = \inf(\sigma_1\cdots\sigma_n(10\overline{1}))$, which are not in the intersection.

The proof for $(1\overline{0}, \overline{1}) = \bigcup_{\boldsymbol{\sigma} \in \{L,M,R\}^\infty} J_{\boldsymbol{\sigma}}$ is similar, with
\begin{multline*}
(\sup(\sigma(01\overline{0})), \sup(\sigma(0\overline{1}))) =  (\sup(\sigma L(01\overline{0})), \sup(\sigma L(0\overline{1}))) \\
\cup [\sup(\sigma M(\overline{0})), \sup(\sigma M(01\overline{0}))] \cup  (\sup(\sigma M(01\overline{0})), \sup(\sigma M(0\overline{1}))) \\
\cup \{\sup(\sigma M(0\overline{1}))\}  \cup (\sup(\sigma R(01\overline{0})), \sup(\sigma R(0\overline{1}))).
\end{multline*}
Hence, the $J_{\boldsymbol{\sigma}}$ are also ordered by the lexicographical order on $\{L,M,R\}^\infty$.
\end{proof}

\begin{proposition} \label{p:cover}
We have the partition
\[
(1,\mu_{0\overline{1}}) = \{\mu_\mathbf{u}:\, \mathbf{u} \in \mathcal{S}_{\{L,M,R\}}\} \cup \bigcup_{\sigma\in\{L,M,R\}^*M} \big([\mu_{\sigma(1\overline{0})}, \mu_{\sigma(01\overline{0})}] \cup [\mu_{\sigma(10\overline{1})}, \mu_{\sigma(0\overline{1})}]\big).
\]
\end{proposition}

\begin{proof}
For $m \in (1,\mu_{0\overline{1}})$, $\boldsymbol{\sigma} \in \{L,M,R\}^\infty$, let
%\begin{align*}
\[
I'_{\boldsymbol{\sigma}}(m) = \begin{cases}\{g_\mathbf{u}(m):\, \mathbf{u}\ \mbox{is a limit word of}\ \boldsymbol{\sigma}\} & \mbox{if $\boldsymbol{\sigma}$ is primitive}, \\ \{g_{\sigma_1\sigma_2 \cdots \sigma_n(1\overline{0})}(m)\} & \mbox{if}\ \sigma_n\sigma_{n+1}\cdots = M\overline{L},\, n\ge1, \\ [g_{\sigma_1\sigma_2 \cdots \sigma_n(\overline{1})}(m), g_{\sigma_1\sigma_2 \cdots \sigma_n(10\overline{1})}(m)] & \mbox{if}\ \sigma_n\sigma_{n+1}\cdots = M\overline{R},\, n\ge1, \\ \emptyset & \mbox{otherwise},\end{cases}
\]
\[
J'_{\boldsymbol{\sigma}}(m) = \begin{cases}\{f_{\mathbf{u}}(m):\, \mathbf{u}\ \mbox{is a limit word of}\ \boldsymbol{\sigma}\} & \mbox{if $\boldsymbol{\sigma}$ is primitive}, \\ [f_{\sigma_1\sigma_2 \cdots \sigma_n(\overline{0})}(m), f_{\sigma_1\sigma_2 \cdots \sigma_n(01\overline{0})}(m)] & \mbox{if}\ \sigma_n\sigma_{n+1}\cdots = M\overline{L}, n\ge1, \\ \{f_{\sigma_1\sigma_2 \cdots \sigma_n(0\overline{1})}(m)\} & \mbox{if}\ \sigma_n\sigma_{n+1}\cdots = M\overline{R},\, n\ge1, \\ \emptyset & \mbox{otherwise}.\end{cases} 
\]
%\end{align*}
By Lemmas~\ref{l:fg} and~\ref{l:partition}, we have 
\[
(1,g_{1\overline{0}}(m)) = \bigcup_{\boldsymbol{\sigma}\in\{L,M,R\}^\infty} I'_{\boldsymbol{\sigma}}(m) \quad \mbox{and} \quad (1,f_{0\overline{1}}(m)) = \bigcup_{\boldsymbol{\sigma}\in\{L,M,R\}^\infty} J'_{\boldsymbol{\sigma}}(m).
\]
(Note that $f_\mathbf{u}(m)$ is close to $f_{\mathbf{u}'}(m)$ if $\sup(\mathbf{u})$ is close to $\sup(\mathbf{u}')$, $g_\mathbf{u}(m)$ is close to $g_{\mathbf{u}'}(m)$ if $\inf(\mathbf{u})$ is close to $\inf(\mathbf{u}')$.)
If $\boldsymbol{\sigma} < \boldsymbol{\sigma}'$, then we have $\beta > \beta'$ if $\beta \in I'_{\boldsymbol{\sigma}}(m)$, $2 \le \beta' \in I'_{\boldsymbol{\sigma}'}(m)$, and $\beta < \beta'$ if $2 \le \beta \in J'_{\boldsymbol{\sigma}}(m)$, $\beta' \in J'_{\boldsymbol{\sigma}'}(m)$, by Lemmas~\ref{l:fg} and~\ref{l:partition}.
Since $\max(f_\mathbf{u}(m), g_\mathbf{u}(m)) \ge 2$ for all $\mathbf{u} \in \{0,1\}^\infty$ and $\inf(\sigma M(1\overline{0})) \le \inf(\sigma M(\overline{0}))$, $\sup(\sigma M(0\overline{1})) \ge \sup(\sigma M(\overline{1}))$ for all $\sigma \in \{L,M,R\}^*$, we have $I'_{\boldsymbol{\sigma}}(m) \subset [2,\infty)$ or $J'_{\boldsymbol{\sigma}}(m) \subset [2,\infty)$ for all $\boldsymbol{\sigma} \in \{L,M,R\}^\infty$.
Therefore, we have $I'_{\boldsymbol{\sigma}}(m) \cap J'_{\boldsymbol{\sigma}}(m) \ne \emptyset$ for some $\boldsymbol{\sigma}\in\{L,M,R\}^\infty$.
If $\boldsymbol{\sigma}$ is primitive, this means that $m = \mu_\mathbf{u}$. 
%Otherwise, let $\sigma' = \sigma_1\sigma_2 \cdots \sigma_n$.
If $\sigma_n\sigma_{n+1}\cdots = M\overline{L}$, then we have $g_{\sigma_1\cdots\sigma_n(1\overline{0})}(m) \in [f_{\sigma_1\cdots\sigma_n(\overline{0})}(m), f_{\sigma_1\cdots\sigma_n(01\overline{0})}(m)]$, which means that $m \in [\mu_{\sigma_1\cdots\sigma_n(1\overline{0})}, \mu_{\sigma_1\cdots\sigma_n(01\overline{0})}]$, see Figure~\ref{f:musigma0}.
Similarly, if $\sigma_n\sigma_{n+1}\cdots = M\overline{R}$, then we have that $m \in [\mu_{\sigma_1\cdots\sigma_n(10\overline{1})}, \mu_{\sigma_1\cdots\sigma_n(0\overline{1})}]$. 
\end{proof}

\begin{proof}[Proof of Theorem~\ref{t:main}]
This is a direct consequence of Propositions~\ref{p:upper}, \ref{p:lower} and~\ref{p:cover}. 
\end{proof}

\section{Final remarks and open questions}
By \cite{KomornikLaiPedicini11,BakerSteiner17,Kwon18}, there are simple formulas for $\mu_{\sigma(1\overline{0})}$, $\mu_{\sigma(\overline{0})}$ and $\mu_{\sigma(0\overline{1})}$, $\sigma \in \{L,R\}^*M$, and for $\mu_\mathbf{u}$, $\mathbf{u} \in \mathcal{S}_{L,R}$. 
This is because, for $\mathbf{u} \in \{\sigma(1\overline{0}), \sigma(0\overline{1})\}$, $\sigma \in \{L,R\}^*M$, or $\mathbf{u} \in \mathcal{S}_{L,R}$, we have $\inf(\mathbf{u}) = 0\mathbf{v}$, $\sup(\mathbf{u}) = 1\mathbf{v}$ for some $\mathbf{v}$, thus $(\beta-1) (1+\pi_\beta(0\mathbf{v})) = (\beta-1)^2 = \beta \pi_\beta(1\mathbf{v})$, where $\beta > 1$ is defined by $\pi_\beta(20\mathbf{v}) = 1$, which gives that $\mu_\mathbf{u} = (\beta-1)^2$. 
For $\mathbf{u} = \sigma(\overline{0})$, we have $\inf(\mathbf{u}) = \overline{0w1}$, $\sup(\mathbf{u}) = \overline{1w0}$, with $\sigma(0) = 0w1$, and
\[
(\beta-1) (1+\pi_\beta(\overline{0w1})) = (\beta-1)\beta \pi_\beta(\overline{10w}) = \frac{(\beta-1)^2\beta^{|\sigma(0)|}}{\beta^{|\sigma(0)|}-1} =  \beta \pi_\beta(\overline{1w0}),
\]
where $\beta > 1$ is defined by $\pi_\beta(20w\overline{0}) = 1$ and $|\sigma(0)|$ is the length of $\sigma(0)$, hence $\mu_{\sigma(\overline{0})} = (\beta-1)^2\beta^{|\sigma(0)|}/(\beta^{|\sigma(0)|}-1)$.
Are there similar formulas for $\sigma \in \{L,M,R\}^*M$?
 
In \cite{BakerSteiner17,Kwon18}, it was proved that the Hausdorff dimension of $\{\mu_\mathbf{u}:\, \mathbf{u} \in \mathcal{S}_{L,R}\}$ is $0$, using that the number of balanced words grows polynomially. 
What is the complexity of $\mathcal{S}_{L,M,R}$?

As mentioned in the Introduction, we know the generalised Komornik--Loreti constant $\mathcal{K}(m)$ only for $m=2$ and when $\mathcal{G}(m) = 1+\sqrt{m} = \mathcal{K}(m) = \mathcal{L}(m)$. 
This is due to the fact that it is usually difficult to study maps with two holes; see Figure~\ref{f:T}.
(For $m=2$, we can use the symmetry of the map~$T$, and for $\mathcal{L}(m) = 1+\sqrt{m}$, we can restrict to sequences in $\{0,1\}^\infty$. )
New ideas are needed for the general case.

Finally, \emph{Sturmian holes} are key ingredients in \cite{Sidorov14}, where supercritical holes for the doubling map are studied. 
Do our Thue--Morse--Sturmian sequences also play a role in this context? 
\bibliographystyle{amsalpha}
\bibliography{criticalbases}

\providecommand{\bysame}{\leavevmode\hbox to3em{\hrulefill}\thinspace}
\providecommand{\MR}{\relax\ifhmode\unskip\space\fi MR }
\providecommand{\MRhref}[2]{%
  \href{http://www.ams.org/mathscinet-getitem?mr=#1}{#2}
}
\providecommand{\href}[2]{#2}
\begin{thebibliography}{KLLdV17}

\bibitem[ABBK19]{AlcarazBarreraBakerKong19}
R.~Alcaraz~Barrera, S.~Baker, and D.~Kong, \emph{Entropy, topological
  transitivity, and dimensional properties of unique {$q$}-expansions}, Trans.
  Amer. Math. Soc. \textbf{371} (2019), no.~5, 3209--3258.

\bibitem[AF09]{AlloucheFrougny09}
J.-P. Allouche and C.~Frougny, \emph{Univoque numbers and an avatar of
  {T}hue-{M}orse}, Acta Arith. \textbf{136} (2009), no.~4, 319--329.

\bibitem[Bak14]{Baker14}
S.~Baker, \emph{Generalized golden ratios over integer alphabets}, Integers
  \textbf{14} (2014), Paper No. A15, 28.

\bibitem[BD14]{BertheDelecroix14}
V.~Berth\'{e} and V.~Delecroix, \emph{Beyond substitutive dynamical systems:
  {$S$}-adic expansions}, Numeration and substitution 2012, RIMS
  K\^{o}ky\^{u}roku Bessatsu, B46, Res. Inst. Math. Sci. (RIMS), Kyoto, 2014,
  pp.~81--123.

\bibitem[BS17]{BakerSteiner17}
S.~Baker and W.~Steiner, \emph{On the regularity of the generalised golden
  ratio function}, Bull. Lond. Math. Soc. \textbf{49} (2017), no.~1, 58--70.

\bibitem[DK93]{DaroczyKatai93}
Z.~Dar\'{o}czy and I.~K\'{a}tai, \emph{Univoque sequences}, Publ. Math.
  Debrecen \textbf{42} (1993), no.~3-4, 397--407.

\bibitem[dVK09]{deVriesKomornik09}
M.~de~Vries and V.~Komornik, \emph{Unique expansions of real numbers}, Adv.
  Math. \textbf{221} (2009), no.~2, 390--427.

\bibitem[GS01]{GlendinningSidorov01}
P.~Glendinning and N.~Sidorov, \emph{Unique representations of real numbers in
  non-integer bases}, Math. Res. Lett. \textbf{8} (2001), no.~4, 535--543.

\bibitem[Hut81]{Hutchinson81}
J.~E. Hutchinson, \emph{Fractals and self-similarity}, Indiana Univ. Math. J.
  \textbf{30} (1981), no.~5, 713--747.

\bibitem[KKL17]{KomornikKongLi17}
V.~Komornik, D.~Kong, and W.~Li, \emph{Hausdorff dimension of univoque sets and
  devil's staircase}, Adv. Math. \textbf{305} (2017), 165--196.

\bibitem[KLLdV17]{KongLiLudeVries17}
D.~Kong, W.~Li, F.~L\"{u}, and M.~de~Vries, \emph{Univoque bases and
  {H}ausdorff dimension}, Monatsh. Math. \textbf{184} (2017), no.~3, 443--458.

\bibitem[KLP11]{KomornikLaiPedicini11}
V.~Komornik, A.~C. Lai, and M.~Pedicini, \emph{Generalized golden ratios of
  ternary alphabets}, J. Eur. Math. Soc. (JEMS) \textbf{13} (2011), no.~4,
  1113--1146.

\bibitem[KP17]{KomornikPedicini17}
V.~Komornik and M.~Pedicini, \emph{Critical bases for ternary alphabets}, Acta
  Math. Hungar. \textbf{152} (2017), no.~1, 25--57.

\bibitem[Kwo18]{Kwon18}
D.~Kwon, \emph{Sturmian words and {C}antor sets arising from unique expansions
  over ternary alphabets}, Ergodic Theory Dynam. Systems (2018), 1--28.

\bibitem[Lai11]{Lai11}
A.~C. Lai, \emph{Minimal unique expansions with digits in ternary alphabets},
  Indag. Math. (N.S.) \textbf{21} (2011), no.~1-2, 1--15.

\bibitem[Lot02]{Lothaire02}
M.~Lothaire, \emph{Algebraic combinatorics on words}, Encyclopedia of
  Mathematics and its Applications, vol.~90, Cambridge University Press,
  Cambridge, 2002.

\bibitem[Sid14]{Sidorov14}
N.~Sidorov, \emph{Supercritical holes for the doubling map}, Acta Math. Hungar.
  \textbf{143} (2014), no.~2, 298--312.

\end{thebibliography}
\end{document}